\documentclass[11pt]{amsart}

\usepackage{caption}
\DeclareCaptionLabelFormat{period}{Table \thetable.}
\usepackage[matrix,arrow,curve]{xy}
\usepackage{amsmath}
\usepackage{amssymb}
\usepackage{amsthm}
\usepackage{stmaryrd}
\usepackage[matrix,arrow,curve]{xy}
\usepackage{bbm}
\usepackage{float}
\usepackage{graphicx}

\usepackage[colorlinks=true]{hyperref}

\renewcommand{\L}{\mathcal{L}}
\newcommand{\A}{\mathbb{A}}

\newcommand{\B}{\mathbb{B}}
\renewcommand{\O}{\mathcal{O}}
\newcommand{\F}{\mathcal{F}}

\renewcommand{\P}{\mathbb{P}}
\newcommand{\Z}{\mathbb{Z}}
\newcommand{\N}{\mathbb{N}}
\newcommand{\R}{\mathbb{R}}

\newcommand{\FM}{\operatorname{FM}}
\renewcommand{\dim}{\operatorname{dim}}
\newcommand{\Hom}{\operatorname{Hom}}
\newcommand{\Ext}{\operatorname{Ext}}
\newcommand{\RHom}{\operatorname{\textbf{R}Hom}}
\newcommand{\rk}{\operatorname{rk}}
\newcommand{\Proj}{\operatorname{Proj}}
\newcommand{\Coh}{\operatorname{Coh}}

\newcommand{\MCM}{\underline{\operatorname{MCM}}}

\newcommand{\cone}{\operatorname{cone}}

\newcommand{\rad}{\operatorname{rad}\langle -, - \rangle}

\newtheorem{Lem}{Lemma}[section]
\newtheorem{Prop}[Lem]{Proposition}
\newtheorem{Cor}[Lem]{Corollary}
\newtheorem{Thm}[Lem]{Theorem}

\newtheorem*{Def}{Definition}

\def\gr{\operatorname{gr}\!}

\usepackage{fancyhdr}
\usepackage{etoolbox}
\patchcmd{\section}{\scshape}{\bfseries}{}{}
\makeatletter
\renewcommand{\@secnumfont}{\bfseries}
\makeatother

\usepackage{etoolbox}
\makeatletter
\patchcmd{\@settitle}{\uppercasenonmath\@title}{}{}{}
\patchcmd{\@setauthors}{\MakeUppercase}{}{}{}

\makeatletter
\newcommand\xleftrightarrow[2][]{%
  \ext@arrow 9999{\longleftrightarrowfill@}{#1}{#2}}
\newcommand\longleftrightarrowfill@{%
  \arrowfill@\leftarrow\relbar\rightarrow}
\makeatother

\title{Betti numbers of MCM modules over the cone of an elliptic normal curve}
\author{Alexander Pavlov}
\address{MSRI, 17 Gauss Way, Berkeley, CA, 94720, USA}
\email{apavlov@msri.org}
\date{}

 \subjclass[2010]{Primary: 
14H52, 
13C14, 
13D02 , 
Secondary:
14F05 , 
}

\keywords{Elliptic curves, Maximal Cohen-Macaulay modules, Ulrich modules, Koszul modules.}

\begin{document}
\begin{abstract}

We apply Orlov's equivalence to derive formulas for the Betti numbers of maximal Cohen-Macaulay modules over the cone an elliptic curve $(E,x)$ embedded into $\mathbb{P}^{n-1}$, by the full linear system $|\mathcal{O}(nx)|$, for $n>3$. The answers are given in terms of recursive sequences. These results are applied to give a criterion of (Co-)Koszulity. 

In the last two sections of the paper we apply our methods to study the cases $n=1,2$\,. Geometrically these correspond to the embedding of an elliptic curve into a weighted projective space. The singularities of the corresponding cones are called minimal elliptic. They were studied by K.Saito \cite{Saito74}, where he introduced the notation $\widetilde{E_8}$ for $n=1$, $\widetilde{E_7}$ for $n=2$ and $\widetilde{E_6}$ for the cone over a smooth cubic, that is, for the case $n=3$\,. For the singularities $\widetilde{E_7}$ and $\widetilde{E_8}$ we obtain formulas for the Betti numbers and the numerical invariants of MCM modules analogous to the case of a plane cubic.
\end{abstract}

\maketitle

\tableofcontents

\section{Introduction}

Let $k$ be an algebraically closed field of characteristic zero. All rings and varieties are defined over $k$. In \cite{Pavlov2015} we obtained formulas for the Betti numbers of maximal Cohen-Macaulay modules over the cone of a smooth plane cubic. In this paper we extend our results to normal elliptic curves in $\P^n$, $n>2$. For more details on geometry of elliptic normal curves see \cite{Hulek86}. For an elliptic curve in $\P^3$ its homogeneous coordinate ring $R$ is a complete intersection of two quadrics, but for $n>3$ the ring $R$ is no longer a complete intersection, but it is Gorenstein. In this case the structure of Betti tables is more involved, we don't give an explicit formulas for the Betti numbers, but derive a recurrence relations. Solving these recurrence relations one can obtain explicit formulas.

\begin{Thm}
Let $X$ be a smooth projective Calabi-Yau variety, $R$ its homogeneous coordinate ring. Then there is an equivalence of categories 
$$
\Phi : D^b(X) \to D^{\gr}_{Sg}(R).
$$
\end{Thm} 
We want to stress that this is a special case of Orlov's much more general result, but we are going to apply only this special case.

It was shown by R.-O. Buchweitz \cite{Buch87} that the singularity category $D^{\gr}_{Sg}(R)$ is equivalent to the stable category of maximal Cohen-Macaulay modules
$$
D^{\gr}_{Sg}(R) \cong \MCM_{\gr}(R).
$$
We will call the equivalence $\Phi : D^b(X) \to \MCM_{\gr}(R)$ Orlov's equivalence. 

We start with the simple observation that the graded Betti numbers are given by the well known formula
$$
\beta_{i,j}(M)=\dim \Ext^i(M, k(-j))\,.
$$
For an MCM module $M$ rewriting this formula in terms of extension groups $\Ext^i(M, k(-j))$ in the stable category and using the fact that Orlov's equivalence is an equivalence of triangulated categories and thus preserve extension groups we can prove the following result:
\begin{Thm}
The graded Betti numbers of an MCM module $M$ are given by
$$
\beta_{i,j}(M)=\dim \Hom_{D^b(X)}(\Phi^{-1}(M), \sigma^{-j}(\Phi^{-1}(k^{st}))[i]),
$$
where $\sigma$ is an endofunctor of $D^b(X)$ corresponding under Orlov's equivalence to the shift of grading by one $(1)$ of a graded module 
$$
\sigma = \Phi^{-1} \circ (1) \circ \Phi
$$ 
\end{Thm}

In \cite{Pavlov2015} we proved 

\begin{Thm}
There is an isomorphism
$$
\Phi^{-1}(k^{st}) \cong \O[1]
$$
in the category $D^b(E)$\,. The graded Betti numbers $\beta_{i,j}$ of the MCM module $M=\Phi(\F)$ are given by the formula
$$
\beta_{i,j}(M)=\dim \Hom_{D^b(E)}(\F, \sigma^{-j}(\O[1])[i])\,.
$$
\end{Thm}

For more detailed preliminaries see \cite{Pavlov2015} and references therein.

The Betti numbers of an indecomposable MCM module $M=\Phi(\F)$ 
$$
\beta_{i,j}(\Phi(\mathcal{F}))=\dim \Hom_{D^b(E)}(\F, V_{-j}[i]),
$$
are determined by the objects $V_j=\sigma^j(\O[1])$. Now we need to know all iterations of the functor $\sigma$. Let us rewrite these objects as
\begin{itemize}
\item $V_j=K_j[j]$, for $j>0$,
\item $V_j=K_j[j+1]$, for $j \leq 0$,
\end{itemize}
For instance, $V_0 \cong \O[1]$ and $V_1 \cong \O(1)[1]$. Then we have 

\begin{Thm}
The object $K_j$ is a vector bundle. The charges $Z(K_j)=\left(
       \begin{matrix}
         r_j \\
         d_j
       \end{matrix}
     \right)$ 
satisfy the recursion formulas
\vspace{-0.2cm}
\begin{gather*}
r_{j+1}=(n-2)r_j-r_{j-1}, \\
d_{j+1}=(n-2)d_j-d_{j-1},
\end{gather*}
and the sequences start with $r_1=1$, $r_2=n-1$, $d_1=n$, and $d_2=n^2-2n$\,.
For $j \leq 0$ we have the same recursion formulas:
\vspace{-0.2cm}
\begin{gather*}
r_{j-1}=(n-2)r_j-r_{j+1}, \\
d_{j-1}=(n-2)d_j-d_{j+1},
\end{gather*}
with initial conditions $r_0=1$, $r_{-1}=n-1$, $d_0=0$, and $d_{-1}=n$\,.
In particular if $j >0$, $\deg(K_j)=\deg(K_{-j})$, and $\rk(K_{j+1})=\rk(K_{-j})$\,.
\end{Thm}

To find all possible Betti table it is enough to consider only vector bundles, as the following theorem states.

\begin{Thm}
Every indecomposable MCM module over $R_E$ corresponds under Orlov's equivalence $\Phi : D^b(E) \to \MCM_{\gr}(R)$ after some shift $(k)$  to an indecomposable vector bundle, possibly translated by some cohomological degree.
\end{Thm}

Finally the main formula in this case gives us

\begin{Cor}
Let $\F[l] \in D^b(E)$ be an indecomposable object, where $\F$ is a vector bundle, then the Betti numbers of $\Phi(\F)$ are given by the following formulae
\begin{enumerate}
  \item If $j \geq 0$, then $\beta_{i,j}=\begin{cases}
\dim H^0(E, \F^\vee \otimes K_{-j}), &\text{if $i-j=l-1$,} \\
\dim H^1(E, \F^\vee \otimes K_{-j}), &\text{if $i-j=l$.} \\
\end{cases}$
  \item If $j<0$, then $\beta_{i,j}=\begin{cases}
\dim H^0(E, \F^\vee \otimes K_{-j}), &\text{if $i-j=l$,} \\
\dim H^1(E, \F^\vee \otimes K_{-j}), &\text{if $i-j=l+1$.} \\
\end{cases}$
\end{enumerate}
In particular, if the charge of $\F$ is fixed, there are only finitely many possibilities for the Betti tables.
\end{Cor}

All potentially nonzero Betti numbers of a given module $M = \Phi(\F)$
lie on the line $D_0=\{(i,j)\in \Z^2 | j=i\}$, the half-line $U_0=\{(i,j) \in \Z^2 | j=i+1, i \geq -1\}$, and the half-line $B_0=\{(i,j) \in \Z^2 |j=i-1, i \leq 0 \}$\,.

\begin{figure}[H]
\centering
\includegraphics{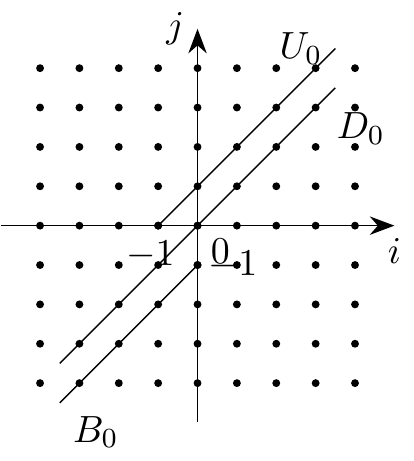}
\captionsetup{labelsep=period}
\caption{Lines $U_0$, $D_0$, $B_0$.}
\end{figure}

As an application of our results we derive a Koszulity criterion for MCM modules over $R$ in terms of vector bundles on the geometric side of Orlov's equivalence. Recall that a module $M$ is called Koszul if $\beta_{i,j}(M)=0$ for $i \neq j$ for $i \geq 0$

\begin{Thm}
An indecomposable MCM module $M=\Phi(\F[l])$ is Koszul if and only if $l=0$ and $\F$ is an indecomposable vector bundle with charge
$Z(\F)=\left(
       \begin{matrix}
         r \\
         d
       \end{matrix}
     \right),$
satisfying the condition $\F \ncong F_r(1)$ and $(r, d)$ satisfying the system of linear inequalities
\begin{align*}
    \left\{
    \begin{aligned}
    r > 0, \\
    d > 0, \\
    d- \mu(rn-d(n-1)) \geq 0,\\
    nr-d \geq 0,
    \end{aligned}
    \right.
\end{align*}
where $\mu=\frac{n-2+\sqrt{n^2-4n}}{2}$\,.
\end{Thm}

Let us illustrate this in the case $n=5$. On the figure 3 the light grey area represents two half-spaces given by the inequalities $\mu(rn-d(n-1))-d \leq 0$ and $nr-d \geq 0$, while the dark grey area represents their intersection.

\begin{figure}[H]
\hspace*{-1cm}
\includegraphics{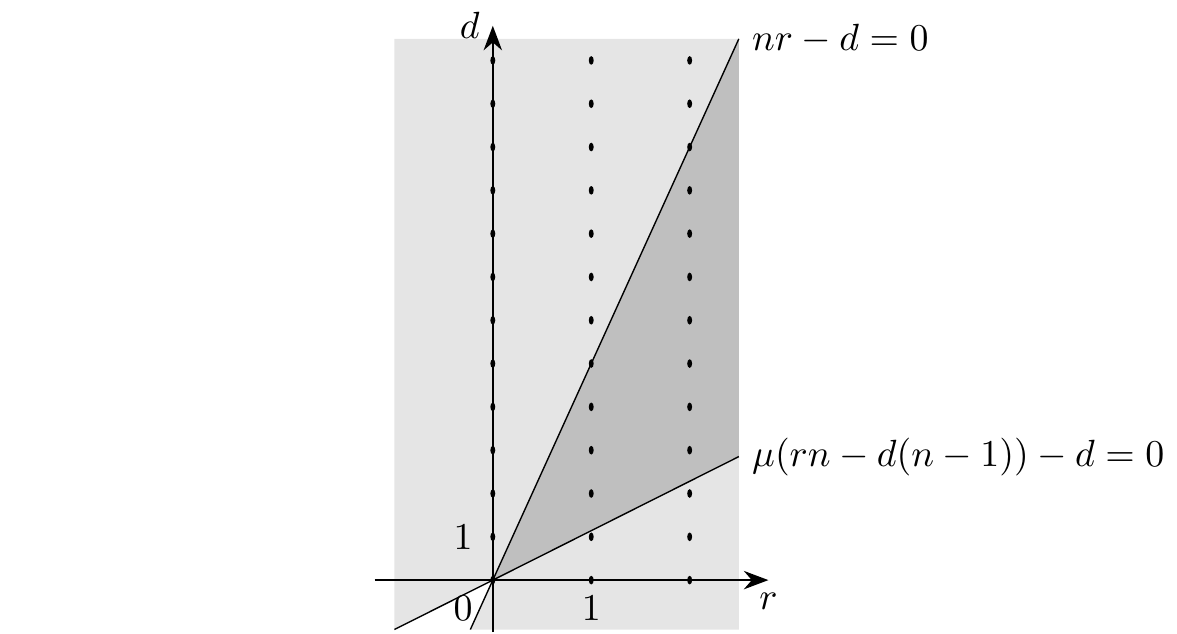}
\captionsetup{labelsep=period}
\caption{Koszul Charges, $n=5$\,.}
\end{figure}

In the last two sections of the paper we presented our results for a normal elliptic curve $E \subset \P^n$. At the end we want to mention that it is also possible to apply our methods to the cases $n=1,2$. In these two cases we obtain homogeneous coordinate rings of an elliptic curve embedded into a weighted projective space. Those rings were extensively studied by K. Saito \cite{Saito74}, under the name $\widetilde{E_8}$ ($n=1$) and $\widetilde{E_7}$ ($n=2$) elliptic singularities. 
The singularity $\widetilde{E_8}$ in Hesse form is
$$
R_{E,1} \cong k[x_0,x_1,x_2]/(x_0^6+x_1^3+x_2^2-3\psi x_0 x_1 x_2),
$$
where $\deg(x_0)=1$, $\deg(x_1)=2$, $\deg(x_2)=3$, and the singularity $\widetilde{E_7}$ has Hesse form
$$
R_{E,2} \cong k[x_0,x_1,x_2]/(x_0^4+x_1^4+x_2^2-3\psi x_0 x_1 x_2),
$$
where $\deg(x_0)=1$, $\deg(x_1)=1$, $\deg(x_2)=2$, the parameter $\psi$ satisfies $\psi^3 \neq 1$.

We show that our approach also works for weighted projective case and obtain explicit answers for Betti numbers of MCM modules over $\widetilde{E_8}$ and $\widetilde{E_7}$ singularities.

\section{Betti numbers of MCM modules}

We fix a line bundle $\L=\O(1)=\mathcal{O}(nx)$, where $n>3$\,. The goal of this section is to describe Betti numbers of MCM modules over the homogeneous coordinate ring
$$
R_E=\bigoplus_{i \geq 0} H^0(E, \L^{\otimes i}),
$$
of $E$ embedded in $\P^{n-1}$ as an elliptic normal curve.

As for the case of a plane cubic we start with the general formula for Betti numbers
$$
\beta_{i,j}(\Phi(\mathcal{F}))=\dim \Hom_{D^b(E)}(\F, V_{-j}[i]),
$$
where as before we set $V_j=\sigma^j(\O[1])$. 

In \cite{KMVdB11} (lemma $4.2.1$) the functor $\sigma$ was completely described. Here we use that description to express $\sigma$ in terms of the spherical twists $\A$ and $\B$\,. The result depends on the degree $n$ of the elliptic normal curve
\begin{Lem}
There is an isomorphism of functors $\sigma \cong \B^n \circ \A\,.$
\end{Lem}

Here $\A$ and $\B$ are spherical twist functors 
\begin{align*}
\mathbb{A}=T_{\O}, & \qquad \mathbb{B}=T_{k(x)}  \cong \O(x) \otimes -
\end{align*}

Let us express object $V_j$ as the translation of an indecomposable sheaf. We consider separately two cases, $j >0$ and $j \leq 0$\,. For $j > 0$ it is easy to see that
$$
V_j = K_j[j],
$$
where $K_j$ is a sheaf. But for $j \leq 0$ the translation is different:
$$
V_j = K_j[j+1],
$$
where again $K_j$ is a sheaf. For instance, $V_0 \cong \O[1]$, hence $K_0=\O$ and $V_1 \cong \O(1)[1]$, hence $K_1 = \O(1)$\,.

The properties of the sheaves $K_j$ for $j \in \Z$ are summarized in the following proposition.
\begin{Prop}
The object $K_j$ is a vector bundle over elliptic curve $E$. If we denote the charge $Z(K_j)=\left(
       \begin{matrix}
         r_j \\
         d_j
       \end{matrix}
     \right)$,
then for $j>0$ rank and degree satisfy the recursion formulae
\begin{gather*}
r_{j+1}=(n-2)r_j-r_{j-1}, \\
d_{j+1}=(n-2)d_j-d_{j-1},
\end{gather*}
and the sequences start with $r_1=1$, $r_2=n-1$, $d_1=n$ and $d_2=n^2-2n$\,.
For $j \leq 0$ we have the same recursion formulae:
\begin{gather*}
r_{j-1}=(n-2)r_j-r_{j+1}, \\
d_{j-1}=(n-2)d_j-d_{j+1},
\end{gather*}
and these sequences start with $r_0=1$, $r_{-1}=n-1$, $d_0=0$ and $d_{-1}=n$\,.
In particular if $j >0$, $\deg(K_j)=\deg(K_{-j})$, and $\rk(K_{j+1})=\rk(K_{-j})$\,.
\end{Prop}

\begin{proof}
Note that $\sigma=\mathbb{B}^n \circ \mathbb{A}$, and the induced transformation of $K_0(E)/\rad$ is
$$
c_n=B^n \circ A=\left(
                       \begin{array}{cc}
                         1 & -1 \\
                         n & 1-n \\
                       \end{array}
                \right)\,.
$$
First consider the case $j>0$ and use induction, where the base $j=1$ of the induction is clear. We see that $r_1=1$ and $d_1=n$, and we can also directly compute $r_2$ and $d_2$\,.
$$
\left(\begin{matrix}
r_2 \\
d_2
\end{matrix}
\right) = Z(K_2) = Z(V_2) = c_n Z(V_1) = c_n \left(
                       \begin{matrix}
                         -1 \\
                         -n
                       \end{matrix}
                \right) =
                \left(
                       \begin{matrix}
                         n-1 \\
                         n^2-2n
                       \end{matrix}
                \right)\,.
$$
Consequently, $d_1>0$ and it is easy to check that $$\frac{n-2+\sqrt{n^2-4n}}{2}d_2>d_1\,.$$
Lemma \ref{seq} below guarantees that $d_j$ is an increasing sequence of positive numbers.

Now we are ready to prove the induction step. Suppose that the statement is true for $j$, then we have an exact triangle
$$
\RHom(\O, K_j)\otimes \O \to K_j \to C_{j+1} \to \ldots,
$$
where $C_{j+1}$ is a cone of the evaluation map. By the induction step we know that the degree of $K_j$ is high enough so that
$$
\RHom(\O, K_j) \cong H^0(E, K_j)[0]\,.
$$
Moreover, by \cite[Lemma 8]{Atiyah57} the evaluation map is surjective, therefore, there is an exact sequence
$$
0 \to K_{j+1}' \to H^0(E, K_j) \otimes \O \to K_j \to 0
$$
of vector bundles on $E$\,. This implies that $C_{j+1} \cong K_{j+1}'[1]$ is the translation of a vector bundle, but
$$
V_{j+1}=C_{j+1} \otimes \O(nx) [j] \cong K_{j+1}' \otimes \O(1) [j+1]=K_{j+1} [j+1],
$$
where we define $K_{j+1}$ as $K_{j+1}=K_{j+1}' \otimes \O(1)$\,.

Observe that
$$
Z(V_j)=(-1)^jZ(K_j)\,.
$$
Therefore,
$$
\left(
       \begin{matrix}
         r_j \\
         d_j
       \end{matrix}
     \right)=
(-c_n)^j\left(
       \begin{matrix}
         -1 \\
         0
       \end{matrix}
     \right)=
     \left(
       \begin{matrix}
         1 & 1 \\
         -n & n-1
       \end{matrix}
     \right)^j
     \left(
       \begin{matrix}
         -1 \\
         0
       \end{matrix}
     \right),
$$
that is equivalent to the given recursion formula.

The case of $j \leq 0$ is completely analogous, but we use a dual sequence of vector bundles on the induction step and the recursion formula follows from
$$
\left(
       \begin{matrix}
         r_{-j} \\
         d_{-j}
       \end{matrix}
     \right)=
(-c_n^{-1})^j\left(
       \begin{matrix}
         1 \\
         0
       \end{matrix}
     \right)=
     \left(
       \begin{matrix}
         n-1 & -1 \\
         n & -1
       \end{matrix}
     \right)^j \left(
       \begin{matrix}
         1 \\
         0
       \end{matrix}
     \right)\,.
$$
\end{proof}

Note that although for $j \leq 0$ and $j >0$ the recursive sequences are given by the same formulae we can not unify them into one sequence because of the difference in the initial values.

\begin{Lem}\label{seq}
Let $\{s_j\}_{j=1}^\infty $ be a recursive sequence defined by
$$
s_{j+1}=(n-2)s_j-s_{j-1},
$$
and denote $\mu=\frac{n-2+\sqrt{n^2-4n}}{2}$\,. Then
\begin{enumerate}
  \item $s_1 > 0$ and $\mu s_2 - s_1 \geq 0$ if and only if the sequence $\{s_j\}$ is an increasing sequence of positive numbers.
  \item $s_1 < 0$ and $\mu s_2 - s_1 \leq 0$ if and only if the sequence $\{s_j\}$ is a decreasing sequence of negative numbers.
\end{enumerate}
\end{Lem}
\begin{proof}
The characteristic equation of the recursive sequence is
$$
\lambda^2-(n-2)\lambda+1=0\,.
$$
The roots of the equation are $\mu=\frac{n-2+\sqrt{n^2-4n}}{2}>1$ and $\mu^{-1}<1$ if $n>4$, while for $n=4$ the equation has a root of multiplicity two at the point $\lambda=1$\,. If $n>4$ the general solution of the recursive sequence is given by the formula
$$
s_j=A \mu^j+B \mu^{-j},
$$
where $A$ and $B$ are some constants uniquely determined by $s_1$ and $s_2$ by the following formulae:
\begin{gather*}
A=\frac{\mu s_2-s_1}{\mu^2-1}\frac{1}{\mu}, \\
B=\frac{\mu^2s_1-s_2}{\mu^2-1}\mu^2\,.
\end{gather*}
If $n=4$, then the general solution is
$$
s_j=Aj+B,
$$
where $A=s_2-s_1$, $B=2s_1-s_2$\,.
Now the lemma easily follows from these explicit formulae.
\end{proof}

When we study Betti tables of graded modules it is enough to describe all possibilities up to a shift of internal degree $(1)$\,. Therefore, the following proposition is very useful, because it shows that it is enough to consider vector bundles and their translations on $E$\,.
\begin{Prop}
Every indecomposable MCM module over $R_E$ after some shift $(1)$ corresponds under Orlov's equivalence $\Phi : D^b(E) \to \underline{MCM}_{\gr}(R_E)$ to an indecomposable vector bundle, possibly translated into some cohomological degree.
\end{Prop}
\begin{proof}
The classification of the indecomposable objects of $D^b(E)$ implies that it is enough to show that in the orbit of the functor $\sigma$ there is at least one object with non-zero rank, but the last statement follows from the explicit recursion relation for ranks.
\end{proof}

If $\F$ is a vector bundle over $E$, it is clear that $\Hom_{D^b(E)}(\O, \F[i])$ can be non zero for at most two values of the parameter $i$\,. Indeed,
$$
\Hom_{D^b(E)}(\O, \F[i])=\begin{cases}
\dim H^0(E, \F), &\text{if $i=0$,} \\
\dim H^1(E, \F), &\text{if $i=1$,} \\
$0$, & \text{otherwise.}
\end{cases}
$$
We apply this observation to the general formula for the Betti numbers and get the following corollary.
\begin{Cor}\label{betti}
Let $\F[l] \in D^b(E)$ be an indecomposable object, where $\F$ is a vector bundle, then the Betti numbers of $\Phi(\F)$ are given by the following formulae
\begin{enumerate}
  \item If $j \geq 0$, then $\beta_{i,j}=\begin{cases}
\dim H^0(E, \F^\vee \otimes K_{-j}), &\text{if $i-j=l-1$,} \\
\dim H^1(E, \F^\vee \otimes K_{-j}), &\text{if $i-j=l$.} \\
\end{cases}$
  \item If $j<0$, then $\beta_{i,j}=\begin{cases}
\dim H^0(E, \F^\vee \otimes K_{-j}), &\text{if $i-j=l$,} \\
\dim H^1(E, \F^\vee \otimes K_{-j}), &\text{if $i-j=l+1$.} \\
\end{cases}$
\end{enumerate}
In particular, if the two discrete parameters $r$ and $d$ are fixed, there are only finitely many possibilities for the Betti tables of graded MCM modules over the ring $R_E$\,.
\end{Cor}
\begin{proof}
Immediately follows from previous results.
\end{proof}

Note that now we have much more concrete formulae for the Betti numbers, because if the charge of $\F$ is fixed, the cohomology groups can be easily calculated, and we have at most two possibilities for each answer depending on whether $\F^\vee \otimes K_{-j} \cong F_{\rk(\F)d_{-j}}$, here $F_{\rk(\F)d_{-j}}$ denotes Atiyah's bundle of rank $\rk(\F)d_{-j}$\,.

\begin{Cor}
Let $i\gg 0$\,. Then the Betti numbers $\beta_{i,j}$ of a finitely generated graded module over $R_E$ grow linearly if $n=4$ and exponentially if $n>4$\,.
\end{Cor}
\begin{proof}
First, by the depth lemma we know that after finitely many steps in a projective resolution the syzygy modules become MCM, thus it is enough to know the asymptotical behavior of the Betti numbers of MCM modules. The condition $i \gg 0$ implies $j \gg 0$ because $i$ and $j$ differs by some constant. The previous Corollary and the formula for the dimensions of cohomology groups for a vector bundle over an elliptic curve
give us
$$
\beta_{i,j} \sim \deg(\F^\vee \otimes K_{-j})\,.
$$
If we fix the charge $Z(\F)=\left(
       \begin{matrix}
         p \\
         q
       \end{matrix}
     \right)$, then we get
$$
\beta_{i,j} \sim pd_{-j}-qr_{-j}\,.
$$
Therefore, the result follows from explicit the formulae in the proof of Lemma \ref{seq}.
\end{proof}

Our next goal is to describe the possible shapes of the Betti tables. We will see that all potentially nonzero Betti numbers
lie on the line $D_l=\{(i,j)\in \Z^2 | j=i-l\}$, the half-line $U_l=\{(i,j) \in \Z^2 | j=i-l+1, i \geq -1\}$, and the half-line $B_l=\{(i,j) \in \Z^2 |j=i-l-1, i \leq 0 \}$\,.

\begin{figure}[H]
\centering
\includegraphics{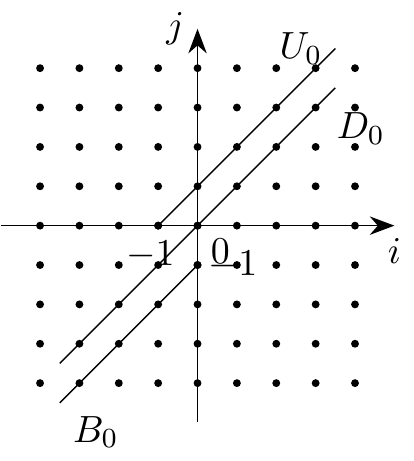}
\captionsetup{labelsep=period}
\caption{Lines $U_0$, $D_0$, $B_0$.}
\end{figure}

From the figure we see that the differentials have only linear and quadratic components, with the possible exception at $i=0$.

Because the calculation of the Betti numbers was reduced to the calculation of dimensions of cohomology groups of $\F^\vee \otimes K_{-j}$, and on an elliptic curve dimensions of cohomology group are almost(except in the special case of Atiyah bundles) defined by degrees, we are going to study the associated sequence of degrees. For a vector bundle $\F$ with charge $Z(\F)=\left(
       \begin{matrix}
         p \\
         q
       \end{matrix}
     \right)$
we define sequence of integer numbers by the following formula:
$$
s_{j}(\F)=\deg(\F^\vee \otimes K_{j})=pd_{j}-qr_{j},
$$
where $j \in \Z$\,. We will write just $s_j$ if $\F$ is understood. The sequence $\{s_j\}$ satisfies the same recursive relation that can be checked directly:
\begin{gather*}
(n-2)s_j-s_{j-1}=(n-2)(pd_{-j}-q r_{-j})-pd_{-j+1}+q r_{-j+1}= \\
=p((n-2)d_{-j}-d_{-j+1})-q((n-2)r_{-j}-r_{-j+1})=pd_{-j-1}-q r_{-j-1}=s_{j+1}\,.
\end{gather*}
Moreover, the formulae for degrees and ranks give us the initial values of the recursive sequences
\begin{gather*}
s_2= p (n^2-2n)-q (n-1),\\
s_1= p n-q,\\
s_0= -q, \\
s_{-1}=p n-q(n-1)\,. \\
\end{gather*}
As in the case of the degrees of the $K_j$ we have the same formula for $s_j$ for $j>0$ and $j \leq 0$, but two initial values need to be modified by direct computation, for instance $s_{-1} \neq (n-2)s_0-s_1$\,. Note that all numbers $s_j$ are linear functions of the variables $p$ and $q$\,. When we write some condition for $s_j(\F)$, we mean that condition on the charge of the vector bundle $\F$\,. The properties of the Betti numbers can be extracted from the properties of the sequence $s_j$ by some elementary analytical calculations.

\begin{Prop}
For fixed $i \in \Z$ the Betti numbers $\beta_{i, j}$ are non zero for only one value of $j$, except for finitely many values of $i$\,. In particular, if $|i|$ is big enough, non-vanishing Betti numbers concentrate along only one of the lines $U_l$ or $D_l$ for $i>0$ and $D_l$ or $B_l$ for $i<0$\,. Moreover, at most one jump between the lines $U_l$ and $D_l$ or $D_l$ and $B_l$ can occur.
\end{Prop}
\begin{proof}
We start with the generic case $n>4$\,. Because we modify the initial values it makes sense to consider the sequences $\{s_j\}_{j>0}$ and $\{s_j\}_{j \leq 0}$ separately. For $j>0$ we have
$$
s_j=A_+ \mu^j+B_+ \mu^{-j},
$$
where
\begin{align*}
A_+=\frac{\mu s_2-s_1}{\mu^2-1}\frac{1}{\mu}, &\quad B_+=\frac{\mu^2s_1-s_2}{\mu^2-1}\mu^2\,.
\end{align*}
For $j \leq 0$
$$
s_j=A_- \mu^{j}+B_- \mu^{-j},
$$
where
\begin{align*}
A_{-}=\frac{\mu s_0-s_1}{\mu^2-1}\mu, & \quad B_{-}=\frac{\mu s_{-1}-s_0}{\mu^2-1}\,.
\end{align*}
Thus for $j \gg 0$ approximately $s_j \sim A_+ \mu^j$, and for $j \ll 0$ approximately $s_j \sim B_{-} \mu^{-j}$\,.
For example, if $A_+$ is positive and $j >0$, then after finitely many steps $s_j$ becomes an increasing sequence of positive numbers. It means that all non-vanishing Betti numbers lie on the line $U_l$\,. The other cases can be treated similarly.

A jump of the Betti numbers corresponds to a change of sign in the sequence $s_j$, therefore, for some $j \in \mathbb{R}$ we should have $s_j=0$, and such an equation can be solved explicitly. For $j > 0$ the solution is given by the formula
$$
j=\frac{1}{2} \log_\mu \left(-\frac{B_+}{A_+}\right)=\frac{1}{2} \log_\mu \left(\frac{\mu^2 s_1-s_2}{s_1- \mu s_2}\right)+\frac{3}{2},
$$
and a solution exists if and only if
\begin{align*}
    \left\{
    \begin{aligned}
    s_1 & > \mu s_2 \\
    \mu^2 s_1-s_2 &> \mu^{-3}(s_1-\mu s_2),
    \end{aligned}
    \right.
    &\quad \text{or} \quad
    \left\{
    \begin{aligned}
    s_1 &< \mu s_2 \\
    \mu^2 s_1-s_2 &< \mu^{-3}(s_1-\mu s_2)\,.
    \end{aligned}
    \right.
\end{align*}
These conditions are linear in the variables $p$ and $q$, thus the plane $K_0(E)/\rad \otimes \R$ is divided into four chambers by two lines $s_1-\mu s_2=0$ and $\mu^2 s_1-s_2 - \mu^{-3}(s_1-\mu s_2)=0$\,. Moreover, we work under the assumption that $\F$ is a vector bundle, that is $p > 0$\,. If $(p, q)$ is an integer point in one of the two chambers described by the system of inequalities above and $p>0$, then the Betti numbers jump between the lines $U_l$ or $D_l$\,.

For $j \leq 0$ a solution of the equation is
$$
j=\frac{1}{2} \log_\mu \left(-\frac{B_{-}}{A_{-}}\right)=\frac{1}{2}\log_\mu \left( \frac{s_0-\mu s_{-1}}{\mu s_0-s_{-1}}\right)- \frac{1}{2}\,.
$$
First note that $s_0=0$ is an obvious solution, and the condition $s_0=0$ is equivalent to $j=0$\,. In terms of the parameters $p$ and $q$ it means that for a vector bundle $\F$ such that the degree is zero, $s_0=q=\deg(\F)=0$, a jump always happens at the $j=0$ level. Other solutions exist if and only if
\begin{align*}
    \left\{
    \begin{aligned}
    s_0 - \mu s_{-1} & > 0 \\
    s_0 & < 0,
    \end{aligned}
    \right.
    &\quad \text{or} \quad
    \left\{
    \begin{aligned}
    s_0 - \mu s_{-1} & < 0 \\
    s_0 & > 0\,.
    \end{aligned}
    \right.
\end{align*}
The description of such solutions in terms of the geometry of $K_0(E) \otimes \R$ is similar to the case $j >0$\,.

For $n=4$ we repeat the same argument, but instead of logarithmic conditions we get linear conditions. For completeness of exposition we derive inequalities that describe the jumps of the Betti numbers. For $j > 0$
$$
s_j=A_{+}j+B_{+},
$$
where
\begin{align*}
A_{+}=s_2-s_1, & \quad B_{+}=2s_1-s_2\,.
\end{align*}
The same argument as before shows that the conditions for a jump are
\begin{align*}
    \left\{
    \begin{aligned}
    s_2-s_1 &< 0 \\
    s_2-2s_1 & < 0,
    \end{aligned}
    \right.
    &\quad \text{or} \quad
    \left\{
    \begin{aligned}
    s_2-s_1 & >0 \\
    s_2-2s_1 & > 0\,.
    \end{aligned}
    \right.
\end{align*}

Finally, for $ j \leq 0$
$$
s_j=A_{-}j+B_{-},
$$
where
\begin{align*}
A_{+}=s_0-s_{-1}, & \quad B_{-}=s_0\,.
\end{align*}
Again we consider the case $s_0=0$ (or equivalently $j=0$) separately, and the conclusion in this case is the same as for $n>4$\,. For the other cases the following condition is equivalent to a jump:
 \begin{align*}
    \left\{
    \begin{aligned}
    s_0 &< 0 \\
    s_0 -s_1 & < 0,
    \end{aligned}
    \right.
    &\quad \text{or} \quad
    \left\{
    \begin{aligned}
    s_0 & > 0 \\
    s_0 -s_1 & > 0\,.
    \end{aligned}
    \right.
\end{align*}
\end{proof}
Note that if $j$, given by the previous formulae, is an integer number and $\F^\vee \otimes K_{-j} \cong F_{pr_{-j}}$ at a jump point, then there are two non-vanishing Betti numbers, both equal to one, otherwise there is only one non-vanishing number located on one of the lines.

\bigskip
\bigskip

\section{Koszulity of MCM modules over the ring \texorpdfstring{$R_E$}{RE}}

In this section we apply the methods developed earlier to derive Koszulity criteria for MCM modules over the ring $R_E$\,. Let us recall the definition of the Koszul and CoKoszul properties.

\begin{Def}
Let $R=\bigoplus_{i \geq 0} R_i$ be a connected graded ring, and let $M$ be a finitely generated graded module over $R$\,. Then
\begin{itemize}
  \item $M$ is called Koszul if $\beta_{i,j}(M)=0$ for $i \neq j$ for $i \geq 0$
  \item $M$ is called Cokoszul if $\beta_{i,j}(M)=0$ for $i \neq j$ for $i \leq 0$
  \end{itemize}
\end{Def}

We start with a Koszulity criterion for MCM modules over $R_E$\,.

\begin{Thm}
An indecomposable MCM module $M=\Phi(\F[l])$ is Koszul if and only if $l=0$ and $\F$ is an indecomposable vector bundle with charge
$Z(\F)=\left(
       \begin{matrix}
         p \\
         q
       \end{matrix}
     \right),$
satisfying the conditions
\begin{itemize}
  \item $s_1 \geq 0$, and if $s_1=0$ then $\F \ncong F_p(1)$,
  \item $s_{0} < 0$, $\mu s_{-1} - s_{0} \leq 0$\,.
\end{itemize}
\end{Thm}
\begin{proof}
Let $M=\Phi(\F[l])$ be an indecomposable Koszul module, then $\F$ is an indecomposable object of $D^b(E)$\,. There are two possibilities for $\F$: $\F$ is a vector bundle, or $\F$ is a skyscraper sheaf. We start with the case of a vector bundle. By the description of the shape of the Betti table given earlier, we conclude that $M=\Phi(\F[l])$ can be Koszul only for $l=0$, or $l=1$\,. First let us assume that $l=1$\,. Corollary \ref{betti} in this case shows that $\beta_{0,-2}=0$ implies $s_2 \geq 0$, and $\beta_{0, -1}=0, \beta_{1,-1}=0$ implies that $s_1=0$\,. Thus we get the conditions
$$
    \left\{
    \begin{aligned}
    s_2=p (n^2-2n)-q (n-1) & \geq 0 \\
    s_1=p n -q & = 0,
    \end{aligned}
    \right.
$$
which imply that $n p \leq 0$, but we assume that $\F$ is a vector bundle, that is $p >0$\,. This contradiction shows that there are no Koszul MCM modules of the form $\Phi(\F[1])$\,.

For modules of the form $\Phi(\F)$ Koszulity is equivalent to two conditions: $\beta_{0, -1}=0$ and $\beta_{i,i+1}=0$ for $i \geq 0$\,. The first condition is equivalent to the condition $s_1 \geq 0$, and if $s_1=0$, we need to assume that $\F^\vee \otimes K_1 \ncong F_{pr_1}$. Note that $r_1=1$ and the last condition is equivalent to
$$
\F \ncong F_p(1).
$$

The second condition is equivalent to $s_{-j} \leq 0$ for $j \geq 1$, where, in principle, for the cases $s_{-j}=0$ we again need the additional assumption that we stay away from Atiyah bundles, but we are going to show that $s_{-j} < 0$\,. Note also that we need to assume $s_0 \leq 0$, because otherwise $\beta_{0,0}=0$, and we get a trivial module. In fact, if $s_0=0$, then $s_1=pn>0$, and we arrive at a contradiction, and, by induction, if $s_{-j}=0$ for some $j \geq 0$, then $s_{-(j+1)} > 0$ which contradicts our assumptions. Applying lemma \ref{seq} we conclude that the second condition is equivalent to $s_0 < 0$ and $s_{-1} \mu - s_0 \leq 0$\,.

Let $\F$ be a skyscraper sheaf. Since $\F$ is indecomposable, it has the form
$$
\F=\mathcal{O}_z/m_z^d,
$$
for some point $z \in E$ and $d \geq 1$\,. Let us apply the functor $\sigma$ to $\F$
$$
\sigma(\F)=\cone(\RHom(\O, \F)\otimes \O \to \F) \otimes \O(nx)\,.
$$
Higher cohomology groups of a skyscraper sheaf vanish, $\RHom(\O, \F) \cong k^d[0]$\,. The evaluation map $k^d \otimes \O \to \mathcal{O}_z/m_z^d$ is surjective, and there is a short exact sequence
$$
0 \to P' \to k^d \otimes \O \to \mathcal{O}_z/m_z^d \to 0,
$$
where $P'$ is the kernel of the evaluation map. Therefore, $\sigma(\F)=P' \otimes \O(nx) [1]$, because $\sigma$ is an autoequivalence and $\F$ is indecomposable $\sigma(\F)$ is indecomposable and $P'$ is also indecomposable. The short exact sequence above implies that the rank of $P'$ is positive, thus $P'$ is an indecomposable vector bundle. It is more convenient to work with the vector bundle
$$
P=P' \otimes \O(nx),
$$
it is easy to see that $\rk(P)=d$ and $\deg(P)=d(n-1)$\,. In terms of the vector bundle $P$ the module $M$ is given by the following isomorphism
$$
\Phi(P[l+1])=M(1)\,.
$$
We reduced now the skyscraper sheaf case to the vector bundle case where all previous results can be applied. The Betti numbers $\beta_{i,j}(M(1))=\beta_{i,j+1}(M)$ of $M(1)$ are only non-zero along the line $j=i-1$, which is only possible for $l=0,1$\,. If $l=0$, one of the conditions $\beta_{0,-2}(M(1))=0$ is equivalent to $s_2(P)>0$, but explicit calculation shows that $s_2(P)=d(1-2n)<0$ if $n>0$, so we arrive at a contradiction. If $l=1$ conditions $\beta_{2,-1}(M(1))=\beta_{1,-1}(M(1))=0$ imply $s_1(P)=0$, but again by explicit calculation $s_1(P)=d$, thus $d=0$, and we arrive at another contradiction. Therefore, the equivalence $\Phi$ never maps skyscraper sheaves to Koszul MCM modules.
\end{proof}

For convenience of the reader we reformulate the conditions of the Theorem in terms of variables $p$ and $q$: $M=\Phi(\F)$ is Koszul if and only if the charge of the vector bundle $\F$ satisfies
\begin{align*}
    \left\{
    \begin{aligned}
    p > 0 \\
    q > 0 \\
    \mu(pn-q(n-1))-q \leq 0\\
    np-q \geq 0
    \end{aligned}
    \right.
\end{align*}
On the plane $(p, q)$ charges of vector bundles corresponding to Koszul modules coincide with the integer points in the closed domain between the two lines $\mu(pn-q(n-1))-q=0$ and $np-q=0$ in the first quadrant.

We illustrate this description with the example $n=5$ on the Figure 2. On this figure the light grey area represents two half-spaces given by inequalities $\mu(pn-q(n-1))-q \leq 0$ and $np-q \geq 0$, while the dark grey area represents their intersection.

\begin{figure}[H]
\centering
\includegraphics{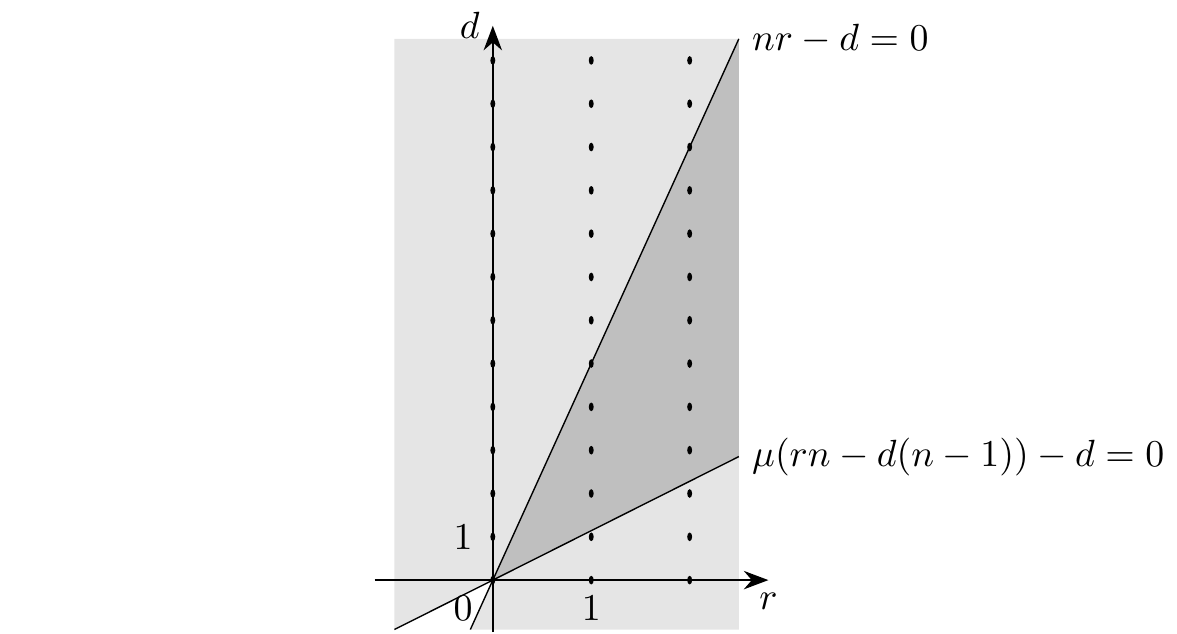}
\captionsetup{labelsep=period}
\caption{Koszul charges, $n=5$\,.}
\end{figure}

The prove of the following Cokoszulity criterion is based on the same case-by-case analysis and does not involve any new ideas.

\begin{Thm}
An indecomposable MCM module $M=\Phi(\F[l])$ is Cokoszul if and only if $l=0$ and $\F$ is an indecomposable vector bundle with charge
$Z(\F)=\left(
       \begin{matrix}
         p \\
         q
       \end{matrix}
     \right),$
satisfying conditions
\begin{itemize}
  \item $s_{-1} \leq 0$, and if $s_{-1}=0$ then $\F \otimes K_{-1}^\vee \ncong F_{p(n-1)}$,
  \item $s_{0} > 0$, $\mu s_{1} - s_{0} \geq 0$\,.
\end{itemize}
\end{Thm}

Our next goal in this section is to compute the Hilbert series for a Koszul module $M=\Phi(\F)$ in terms of the rank $p$ and degree $q$ of a vector bundle $\F$\,. For a Koszul module $M$ we have the following simple relation between Hilbert series $H_M(t)=\sum_i \dim(M_i) t^i$ and Poincar\'{e} series $S(t)=\sum_i \beta_{i,i} t^i$:
$$
H_M(t)=\sum_{i \geq 0} (-1)^i H_{R(-i)^{\beta_{i,i}}}=\sum_{i \geq 0} (-1)^i \beta_{i,i} t^i H_{R_E}(t)=S(-t)H_{R_E}(t)\,.
$$
On the other hand, the relation between the degrees $s_i$ and the Betti numbers in the Koszul case takes the simplest possible form:
$$
\beta_{i,i}(M)=-s_{-i}(F)\,.
$$
Now we can see that the conditions $s_{0} < 0$, $\mu s_{-1} - s_{0} \leq 0$ of the Koszulity criterion mean exactly that $\beta_{i,i} >0$ for $i \geq 0$\,. We use the generating function for the sequence $s_{-j}$ to express the Poincar\'{e} series in the form
$$
S(t)=-\frac{s_0+(s_{-1}-(n-2)s_0) t}{t^2-(n-2)t+1}\,.
$$
This formula for the Poincar\'{e} series together with the formula for the Hilbert series of the ring $R_E$
$$
H_{R_E}=\frac{1+(n-2)t+t^2}{(1-t)^2},
$$
gives us the Hilbert series of $M$
$$
H_M(t)=\frac{-s_0+(s_{-1}-(n-2)s_0) t}{(1-t)^2}=\frac{q+(pn-q)t}{(1-t)^2}\,.
$$
In particular, we can extract the formula for multiplicity: $e(M)=pn$, and interpret the condition $s_1 \geq 0$ of the Koszulity criterion as Ulrich's bound $\mu(M) \leq e(M)$ for the Koszul module $M=\Phi(\F)$\,.

We can apply our results to identify the maximally generated MCM modules.

\begin{Prop}
A Koszul module $M=\Phi(\F)$ is a maximally generated MCM module if and only if the following relation between rank $p$ and degree $q$ of a vector bundle $\F$ holds:
$$
pn=q\,.
$$
The Hilbert series of such a module is given by the formula
$$
H_M(t)=\frac{q}{(1-t)^2},
$$
and the Poincar\'{e} series is
$$
S(t)=\frac{q}{t^2-(n-2)t+1}\,.
$$
\end{Prop}

By \cite[p.44]{PP05} the algebra $R_E$ is Koszul, and by \cite[Theorem 3.4]{IR09} any maximally generated MCM module over a Koszul algebra is a Koszul module. Consequently, the previous Proposition gives the complete description of maximally generated MCM modules over the homogeneous coordinate ring $R_E$ of an elliptic curve $E$\,.

For a maximally generated MCM module $M=\Phi(\F)$ we can give a more explicit description.

\begin{Thm}
Any maximally generated MCM module over $R_E$ is isomorphic to $\Gamma_*(\F)=\oplus_{i \in \Z} H^0(E, \F(i))$, where $\F$ is a vector bundle, whose rank $p$ and degree $q$ of $\F$ satisfy the condition $pn=q$, and $\det(\F) \ncong \O(p)$\,.
\end{Thm}
\begin{proof}
First, we compute the degrees of the Serre twists $\deg(\F(i))=pi\deg(\O(1))+q=q(i+1)$\,. Therefore, $H^0(E, \F(i)) \cong 0$ for $i \leq -1$, where we use $\det(\F) \ncong \O(p)$ for $i=-1$ and $H^1(E, \F(i)) \cong 0$ for $i \geq 0$, thus
$$
\Gamma_*(\F)=\bigoplus_{i \in \Z} H^0(E, \F(i)) \cong \bigoplus_{i \geq 0} \RHom(\O, F(i))\,.
$$
Moreover, as $\Gamma_*(\F)$ is an MCM module if $\F$ is vector bundle, it means that
$$
\Gamma_*(\F) \cong \Phi(\F)\,.
$$
\end{proof}

\section{The Elliptic Singularity \texorpdfstring{$\widetilde{E_8}$}{E8} }

So far we only considered graded rings
$$
R_{E,n}=\oplus_{i \geq 0} H^0(E, \O(inx)),
$$
for $n \geq 3$\,. Now we are going to study two more cases, $n=1,2$\,. For these cases $\O(nx)$ is no longer a very ample line bundle and, consequently, $R_{E,1}$ and $R_{E,2}$ are homogeneous coordinate rings of an elliptic curve embedded into a weighted projective space. These two cases were extensively studied by K. Saito \cite{Saito74}, under the name $\widetilde{E_8}$ ($n=1$) and $\widetilde{E_7}$ ($n=2$) singularities. In this section we consider the $\widetilde{E_8}$ singularity, and the singularity $\widetilde{E_7}$ is considered in the next section.

We start with the equation of the $\widetilde{E_8}$ singularity in Hesse form:
$$
R_{E,1} \cong k[x_0,x_1,x_2]/(x_0^6+x_1^3+x_2^2-3\psi x_0 x_1 x_2),
$$
where $\deg(x_0)=1$, $\deg(x_1)=2$, $\deg(x_2)=3$, and the parameter $\psi$ satisfies $\psi^3 \neq 1$\,.

Our first goal is to show that we have an equivalence of abelian categories $\Proj R_{E,1} \cong \Proj R_{E,3}$\,. Here the notation $\Proj$ denote the same category as $\operatorname{qgr}$\,. Note that the homogeneous coordinate ring of the Hesse cubic $R_{E,3}$ is the third Veronese subring of $R_{E,1}$:
$$
R_{E,3} \cong R_{E,1}^{(3)},
$$
where in general we use $A^{(n)}$ to denote the $n$-th Veronese subalgebra of $A$\,. S.P.Smith in \cite{Smith03} studied the relationship between $\Proj A$ and $\Proj A^{(n)}$, for an algebra $A$ that is not necessarily commutative, nor necessarily generated in degree $1$. We use the following proposition from that paper.
\begin{Prop}\label{Smith}
Let $A$ be an $\N$-graded algebra, and fix $n \geq 2$\,. Denote $I_r=A^{(n)+r}A$ for $r \in \Z$ the indicated ideals in $A$\,. Let
$$
I=\bigcap_{r \in \Z} I_r = I_1 \cap I_2 \cap \ldots \cap I_n\,.
$$
Then there is an isomorphism
$$
\Proj A \backslash Z \cong \Proj A^{(n)} \backslash Z',
$$
where $Z'$ and $Z$ are the loci of $I$ and $I^{(n)}$ respectively.
\end{Prop}
\begin{proof}
This is proposition $4.8$ of \cite{Smith03}.
\end{proof}

We apply this proposition to the case at hand.

\begin{Cor}
There is an equivalence of abelian categories
$$
\Proj R_{E,1} \cong \Proj R_{E,3}\,.
$$
\end{Cor}
\begin{proof}
For the case of $R_E$ it is easy to find generators of the ideals $I_1$, $I_2$, $I_3$ and $I$ explicitly.
One has $I_1= A^{(3)+1}A=(x_0^4, x_0^2 x_1, x_0 x_2, x_1^2)$, and $I_2= A^{(3)+2}A=(x_0^5, x_0^3 x_1, x_0^2 x_2, x_1 x_2)$,
whereas $I_3= A^{(3)}A=(x_0^3, x_0 x_1, x_2)$\,.
The intersection is $I=I_1 \cap I_2 \cap I_3=(x_0^5,x_0^3 x_1, x_0 x_2, x_1^2 x_2)$\,. It is clear that the ideal $I$ is cofinite, and thus so is the ideal $I^{(n)}$. Therefore, $Z=Z'=\emptyset$\,.
\end{proof}

Next we are going to show that the calculations of the Betti numbers and the Hilbert series of MCM modules over the $\widetilde{E_8}$ singularity are almost the same as the calculations for a cone over a plane cubic curve.

The functor $\sigma$ in this case takes the form
$$
\sigma=\B \circ \A,
$$
where $\B=T_{k(x)}=\O(x) \otimes -$, $\A=T_{\mathcal{O}_E}$\,. Because $R_{E,1}$ is a hypersurface of degree $6$ we have a natural isomorphism $[2] \cong \sigma^6$ of functors in $D^b(E)$\,. In \cite{ST01} p. $68$, another isomorphism was established:
$$
\sigma^3 \cong \FM_\mathcal{P}^{-2},
$$
where $\FM_\mathcal{P}$ denotes the autoequivalence of $D^b(E)$ given by the Fourier-Mukai transform with kernel represented by the Poincar\'{e} bundle $\mathcal{P}$\,. If we denote $\iota : E \to E$ the map taking inverses in $E$ as an algebraic group, then we can continue the previous isomorphism:
$$
\sigma^3 \cong \FM_\mathcal{P}^{-2} \cong \iota^* \circ [1]\,.
$$

To calculate Betti numbers we need to calculate $\sigma^j(\O[1])$ for $j \in \Z$, but $\sigma^3(\O[1]) \cong \hat{\iota}^*\O[2] \cong \O[2]$, therefore, the sequence of objects $\sigma^j(\O[1])$ is not $6$-periodic, but already $3$-periodic up to translation $[1]$\,.

\begin{Lem}
We have isomorphisms $\sigma^{-1}(\O[1])\cong k(x)$, $\sigma^{-2}(\O[1])\cong \O(x)$, $\sigma^{-3}(\O[1])\cong \O$, and any other object in the sequence $\sigma^j(\O[1])$ can be determined from the formula
$$
\sigma^{3i+j}(\O[1]) \cong \sigma^j(\O[1])[i],
$$
where $j \in \{-3,-2,-1\}$\.,
\end{Lem}
\begin{proof}
It is clear that $\sigma(\O[1]) \cong \O(x)$, and then $\sigma^2(\O[1])$ is a cone of the following evaluation map:
$$
\sigma^2(\O[1])=\cone(\RHom(\O, \O(x)) \otimes \O \to \O(x))\,.
$$
It is easy to see that $\RHom(\O, \O(x)) \otimes \O \cong \O$, and the evaluation map has cokernel $k(x)=\mathcal{O}_x/m_x$\,.
\end{proof}

We translate properties of $\sigma^{-j}(\O[1])$ into properties of Betti numbers of $M=\Phi(\F)$\,. We see that
\begin{equation*}
\begin{split}
\beta_{i+1,j}&=\beta_{i,j}(M)=\dim \Ext^{i+1}_{D^b(E)}(\F, \sigma^{-j}(\O[1]))=\\
&=\dim \Ext^{i}_{D^b(E)}(\F, \sigma^{-j+3}(\O[1]))=\beta_{i,j-3}\,.
\end{split}
\end{equation*}
In other words, the minimal free resolution of $M$ over $R_{E,1}$ has the form
$$
0\leftarrow M \leftarrow F \leftarrow F(-3) \leftarrow F(-6) \leftarrow \ldots
$$
For a suitable finitely generated module free module $F$. In particular, if we want to compute the Betti table, it is enough to compute only $\beta_{0,*}$\,.

\begin{Prop} If $M=\Phi(\F[l])$, where $\F \in \Coh(D^b)$, then $\beta_{0,j}(M)=0$, except for $-2-3l \leq j \leq 3-3l$\,. In particular, if $l \leq -1$ then $M$ is generated in positive degrees.
\end{Prop}

The functor $\sigma$ induces the linear map
$$
[\sigma]=\begin{pmatrix}
0 & -1 \\
1 & 1
\end{pmatrix}
$$
on $K(E)/\rad$\,. It is of order six, $[\sigma]^6=1$, and we can choose a fundamental domain for this action to be
\begin{gather*}
r>0, \\
0 \leq d < r\,.
\end{gather*}
In this fundamental domain the rank of $\F$ is positive, therefore, $\F$ is a vector bundle. In fact, our previous result can be improved: the module $M=\Phi(\F)$, corresponding to the case $l=0$ in the proposition above, is still generated in positive degrees. This improvement is based on the explicit computation of the Betti numbers in term of the charge of a vector bundle as before. As we compute Betti numbers, we again get several cases that should be treated separately. A summary of these computations is given in the next proposition.

\begin{Prop}
Let $\F$ be a vector bundle with the charge
$
Z(\F)=\left(
       \begin{matrix}
         r \\
         d
       \end{matrix}
     \right)
$
in the fundamental domain. The Betti numbers of the MCM module $M=\Phi(\F)$ can be expressed as
dimensions of the cohomology groups on an elliptic curve in the following way
\begin{align*}
&\beta_{0,0}=\dim H^1(E,\F^\vee)=
\begin{cases}
$1$, &\text{if $\det \F \cong \O$,} \\
$0$, &\text{otherwise.} \\
\end{cases} \\
&\beta_{0,1}=\dim H^0(E, \F^\vee \otimes k(x))=r\,. \\
&\beta_{0,2}=\dim H^0(E,\F^\vee\otimes \O(x))=r-d\,.\\
&\beta_{0,3}=\dim H^0(E,\F^\vee)=
\begin{cases}
$1$, &\text{if $\det \F \cong \O$,} \\
$d$, &\text{otherwise.} \\
\end{cases}
\end{align*}
\end{Prop}
\begin{proof}
The preceding results guarantee that we get zero Betti numbers, $\beta_{0,j}=0$, for $j<-2$ and $j>3$\,. But the explicit calculation of $\beta_{0,-2}$ as
\begin{equation*}
\begin{split}
\beta_{0,-2}&=\dim \Hom(\F,\sigma^2 (\O[1])) = \dim \Hom(\F,\sigma^{-1} (\O[1]) [1])=\\
&= \dim \Hom(\F,k(x)[1]) = \dim H^1(E,\F^\vee\otimes k(x)) = 0,
\end{split}
\end{equation*}
and the explicit calculation of $\beta_{0,-1}$ as
\begin{equation*}
\begin{split}
\beta_{0,-1}&=\dim \Hom(\F,\sigma^1 (\O[1])) = \dim \Hom(\F,\sigma^{-2} (\O[1]) [1])=\\
&=\dim \Hom(\F,\O(x)[1])= \dim H^1(\F^\vee,\O(x))=0,
\end{split}
\end{equation*}
allow to restrict the window of non-trivial Betti numbers to $0 \leq j \leq 3$\,. The computations of the other Betti numbers use our standard tools and are straightforward:

\begin{gather*}
\beta_{0,0}=\dim \Hom(\F, \O[1])=\dim H^1(E,\F^\vee)=
\begin{cases}
$1$, &\text{if $\F \cong F_r$,} \\
$0$, &\text{otherwise.} \\
\end{cases} \\
\begin{split}
\beta_{0,1}&=\dim \Hom(\F,\sigma^{-1} (\O[1]))=\dim\Hom(\F, k(x))= \\
&=\dim H^0(E, \F^\vee \otimes k(x))=r\,.
\end{split} \\
\begin{split}
\beta_{0,2}&=\dim\Hom(\F,\sigma^{-2} (\O[1]))=\dim\Hom(\F,\O(x))= \\
&=\dim H^0(E,\F^\vee\otimes \O(x))=r-d\,.
\end{split} \\
\begin{split}
\beta_{0,3}&=\dim\Hom(\F,\sigma^{-3} (\O[1]))=\dim\Hom(\F,\O)= \\
&=\dim H^0(E,\F^\vee)=
\begin{cases}
$1$, &\text{if $\F \cong F_r$,} \\
$d$, &\text{otherwise.} \\
\end{cases}
\end{split}
\end{gather*}
\end{proof}

From this Proposition we see that there are two cases: in the first $\F$ is the Atiyah bundle, and in the second case $\F$ is generic.

As in the case of a plane cubic we can read off the rank and degree of $\F$ from the minimal free resolution of $M$\,. As before, it is more convenient to present the results of our computations in form of the two possible shapes of the Betti table:

\begin{table}[H]
\begin{center}
\begin{tabular}{|r|r|r||r|r|}
\hline
& \multicolumn{2}{|c||}{$\F \cong F_r$} & \multicolumn{2}{|c|}{Generic $\F$}\\
\hline
$i$ & $i=0$ & $i=1$ & $i=0$ & $i=1$ \\
\hline
$0$ & $1$ & $0$  & $0$   & $0$   \\
$1$ & $r$ & $0$  & $r$   & $0$   \\
$2$ & $r$ & $1$  & $r-d$ & $0$   \\
$3$ & $1$ & $r$  & $d$   & $r$   \\
$4$ & $0$ & $r$  & $0$   & $r-d$ \\
$5$ & $0$ & $1$  & $0$   & $d$   \\
\hline
\end{tabular}
\caption{}
\end{center}
\end{table}

Assume that $M$ is as in the foregoing proposition. Then the Hilbert series is easy to compute:
$$
H_M(t)=\frac{H_F(t)}{1+t^3}=\frac{B(t)H_{R_{E,1}}(t)}{1+t^3},
$$
where $B(t)=\sum_j \beta_{0,j}t^j$, and
$$
H_{R_{E,1}}=\frac{1-t^6}{(1-t)(1-t^2)(1-t^3)}=\frac{1-t+t^2}{(1-t)^2}\,.
$$

The ring $R_{E,1}$ has multiplicity $e(R_{E,1})=2$, because the multiplicity of any hypersurface is equal to the minimum of the degrees of a non-zero monomials in the equation of the hypersurface. The rank of a module can be computed as
$$
\rk M = \lim_{t \to 1} \frac{H_M(t)}{H_R(t)} = \frac{B(1)}{2} = \frac{\mu(M)}{2}\,.
$$
The multiplicities of a ring and a module are related as before by
$$
e(M)=e(R_{E,1}) \rk(M)=2 \rk(M)\,.
$$
We present the numerical information, namely rank, multiplicity and minimal number of generators, of an MCM module $M$ as a table:

\begin{table}[H]
\begin{center}
\begin{tabular}{|r|r|r|r|}
\hline
case & e(M) & $\mu(M)$ & $\rk(M)$\\
\hline
$\F \cong F_r$ & $2(r+1)$ & $2(r+1)$  & $r+1$  \\
\hline
Generic $\F$ & $2r$ & $2r$  & $r$ \\
\hline
\end{tabular}
\caption{}
\end{center}
\end{table}

In particular, we see that every MCM module is maximally generated in this case. More generally, over a ring of multiplicity two any MCM module is maximally generated, this was shown in \cite[Corollary 1.4]{HK87}.

\section{The Elliptic Singularity \texorpdfstring{$\widetilde{E_7}$}{E7}}

In this section we briefly summarize the results for an elliptic singularity of type $\widetilde{E_7}$, as they are only slightly different from the results of the $\widetilde{E_8}$ case. The Hesse form of a $\widetilde{E_7}$ singularity is
$$
R_{E,2} \cong k[x_0,x_1,x_2]/(x_0^4+x_1^4+x_2^2-3\psi x_0 x_1 x_2),
$$
where $\deg(x_0)=1$, $\deg(x_1)=1$, $\deg(x_2)=2$, and the parameter $\psi$ satisfies $\psi^3 \neq 1$\,. Note that as for the $\widetilde{E_8}$ case the multiplicity of $R_{E,2}$ is two.

The result of S.P.Smith quoted in above \ref{Smith} implies the analogous corollary for the abelian category $\Proj R_{E,2}$\,.
\begin{Cor}
There is an equivalence of abelian categories
$$
\Proj R_{E,2} \cong \Proj R_{E,3}\,.
$$
\end{Cor}

In this case, $\sigma^2 \cong \hat{\iota}[1]$, and we have analogously

\begin{Lem}
There are isomorphisms $\sigma^{-1}(\O[1])\cong \O(2x)$, $\sigma^{-2}(\O[1])\cong \O$ and any other object in the sequence $\sigma^j(\O[1])$ can be determined from the formula
$$
\sigma^{2i+j}(\O[1]) \cong \sigma^j(\O[1])[i],
$$
where $j \in \{-2,-1\}$\,.
\end{Lem}

A MCM module $M$ has thus a resolution of the form
$$
0\leftarrow M \leftarrow F \leftarrow F(-2) \leftarrow F(-4) \leftarrow \ldots
$$
For a suitable finitely generated module free module $F$. In particular, it is enough to calculate $\beta_{0,*}$\,.

\begin{Prop} If $M=\Phi(\F[l])$, where $\F \in \Coh(D^b)$, then $\beta_{0,j}(M)=0$, except for $-1-2l \leq j \leq 2-2l$\,. In particular, if $l \leq -1$, then $M$ is generated in positive degrees.
\end{Prop}

The functor $\sigma$ induces the linear map
$$
[\sigma]=\begin{pmatrix}
1 & -1 \\
2 & -1
\end{pmatrix}
$$
on $K_0(E)/\rad$\,. It is of order four, $[\sigma]^4=1$,  and we can choose a fundamental domain for this action to be
\begin{gather*}
r>0, \\
0 \leq d < 2r\,.
\end{gather*}
In such a fundamental domain the rank of $\F$ is positive, therefore, $\F$ is a vector bundle. As before, the result of generation in positive degrees can be extended to the case $l=0$ by explicit computation. A summary of these computations is given in the following proposition.

\begin{Prop}
Let $\mathcal{F}$ be a vector bundle with the charge
$
Z(\mathcal{F})=\left(
       \begin{matrix}
         r \\
         d
       \end{matrix}
     \right)
$
in the fundamental domain. The Betti numbers of an MCM module $M=\Phi(\F)$ can be expressed as
dimensions of cohomology groups on the elliptic curve in the following way:
\begin{align*}
&\beta_{0,0}=\dim H^1(E,\F^\vee)=
\begin{cases}
$1$, &\text{if $\F \cong F_r$,} \\
$0$, &\text{otherwise.} \\
\end{cases} \\
&\beta_{0,1}=\dim H^0(E, \F^\vee \otimes \O(2x))=2r-d\,. \\
&\beta_{0,2}=\dim H^0(E,\F^\vee)=
\begin{cases}
$1$, &\text{if $\F \cong F_r$,} \\
$d$, &\text{otherwise.} \\
\end{cases}
\end{align*}

\end{Prop}
As before, we see that there are two cases: in the first case $\F$ is the Atiyah bundle, and in the second case $\F$ is generic. As in the case of a plane cubic we can read off the rank and degree of $\F$ from the minimal free resolution of $M$\,. We present the two possible shapes of Betti tables below.

\begin{table}[H]
\begin{center}
\begin{tabular}{|r|r|r||r|r|}
\hline
& \multicolumn{2}{|c||}{$\F \cong F_r$} & \multicolumn{2}{|c|}{Generic $\F$}\\
\hline
$i$ & $i=0$ & $i=1$ & $i=0$ & $i=1$ \\
\hline
$0$ & $1$    & $0$     & $0$     & $0$    \\
$1$ & $2r$ & $1$     & $2r-d$  & $0$    \\
$2$ & $1$    & $2r$  & $d$     & $2r-d$ \\
$3$ & $0$    & $1$     & $0$     & $d$    \\
\hline
\end{tabular}
\caption{}
\end{center}
\end{table}

With $M$ as in the proposition above, the Hilbert series is
$$
H_M(t)=\frac{H_F(t)}{1+t^2}=\frac{B(t)H_{R_{E,2}}(t)}{1+t^2},
$$
where $B(t)=\sum_j \beta_{0,j}t^j$, and
$$
H_{R_{E,2}}=\frac{1-t^4}{(1-t)^2(1-t^2)}=\frac{1+t^2}{(1-t)^2}\,.
$$

The computation of the numerical invariants for MCM modules over the ring $R_{E,2}$ carries over verbatim and the table of numerical invariants is exactly the same as for the case $\widetilde{E_8}$\,.

\section*{Acknowledgement}
The results of the paper are part of the results of the author's PhD thesis. I am pleased to thank my advisor Ragnar-Olaf Buchweitz, who posed the problem of computing Betti numbers in this case to me and encouraged me during my work on it. His interest in my results and numerous discussions on all stages of the project always motivated me.

\end{document}